\documentclass[10pt]{article}
\usepackage[dvips]{graphicx}
\usepackage{amssymb,color}
\usepackage[latin1]{inputenc}
\usepackage{epic}
\usepackage{pstricks}
\usepackage{pst-node}
\usepackage{pstricks-add}
\usepackage[latin1]{inputenc}
\usepackage{color}
\usepackage{lettrine}
\usepackage{calc,pifont}
\usepackage[noindentafter,calcwidth]{titlesec}
\usepackage{pstricks}
\usepackage{pst-plot}
\usepackage{pst-node}
\usepackage{amsmath}
\usepackage{amssymb}
\usepackage{amsthm}
\usepackage{subfigure}
\usepackage{lineno}

\def\BBox{\kern  -0.2cm\hbox{\vrule width 0.2cm height 0.2cm}}

\newtheorem{example}{Example}

\newtheorem{teo}{Theorem}[section]
\newtheorem{coro}[teo]{Corollary}
\newtheorem{lema}[teo]{Lemma}
\newtheorem{prop}[teo]{Proposition}
\theoremstyle{definition}

\theoremstyle{remark}

\title{A note on two orthogonal totally $C_4$-free one-factorizations
	of complete graphs}
\author{Adrián Vázquez-Ávila\thanks{adrian.vazquez@unaq.edu.mx}\\
{\small Subdirección de Ingeniería y Posgrado}\\
{\small Universidad Aeronáutica en Querétaro}\\
}

\date{}

\begin{document}
\maketitle
\begin{abstract}
A pair of orthogonal one-factorizations $\mathcal{F}$ and $\mathcal{G}$ of the complete graph $K_n$ is totally $C_4$-free, if the union $F\cup G$, for any $F,G\in\mathcal{F}\cup\mathcal{G}$, does not include a cycle of length four.

In this note, we prove if $q\equiv3$ (mod 4) is a prime power with $q\geq11$, then there is a pair of orthogonal totally $C_4$-free one-factorizations of $K_{q+1}$.
\end{abstract}


\section{Introduction}
An \emph{one-factor} of a graph $G$ is a regular spanning subgraph of degree one. An \emph{one-factorization} of a graph G is a set $\mathcal{F}=\{F_1, F_2,\ldots, F_n\}$ of edge disjoint one-factors such that $E(G)=\displaystyle\cup_{i=1}^nE(F_i)$. Two one-factorizations $\mathcal{F}=\{F_1, F_2,\ldots, F_n\}$ and $\mathcal{H}=\{H_1, H_2,\ldots, H_n\}$ of a graph $G$ are \emph{orthogonal} if $|F\cap H|\leq1$, for every $F\in\mathcal{F}$ and $H\in\mathcal{H}$. A one-factorization $\mathcal{F}=\{F_1, F_2,\ldots, F_n\}$ of a graph $G$ is said to be \emph{$k$-cycle free} if $F_i\cup F_j$ does not contains a cycle of length $k$, $C_k$.

The existence of a $C_4$-free one-factorization of the complete graph $K_n$, for even $n\geq6$, has already been observed by Phelps et. al. in \cite{Phelps}, where they were used to give the existence of simple quadruple systems with index three. 

\begin{teo}\cite{Phelps}
A $C_4$-free one-factorization of complete graphs $K_n$ exists, if and only if, $n$ is even  $n\geq6$.	
\end{teo}

In general, Meszka in \cite{Meszka} proved for each even $n$ and each even $k\geq4$ with $k\not=\frac{n}{2}$, the complete graph $K_n$ has a $C_k$-free one-factorization.

A pair of orthogonal one-factorizations $\mathcal{F}$ and $\mathcal{H}$ of the complete graph $K_n$ is $C_4$-free if $F\cup H$ does not contains a cycle of length four, for all $F\in\mathcal{F}$ and $H\in\mathcal{H}$. On the other hand, a pair of orthogonal one-factorizations $\mathcal{F}$ and $\mathcal{H}$ of complete the graph $K_n$ is \emph{totally} $C_4$-free, if $F\cup H$ does not include a cycle of length four, for all $F,H\in\mathcal{F}\cup\mathcal{H}$ \cite{Bao}. 

An interesting way for constructing one-factorizations of complete graphs is using starters of aditive Abelian groups of odd order: Let $\Gamma$ be a finite additive Abelian group of odd order $n=2k+1$, and let $\Gamma^*=\Gamma\setminus\{0\}$ be the set of non-zero elements of $\Gamma$. A \emph{starter} for $\Gamma$ is a set $S=\{\{x_1,y_1\},\ldots,\{x_k,y_k\}\}$ such that
$\left\{x_1,\ldots,x_k,y_1,\ldots,y_k\right\}=\Gamma^*$
and $\left\{\pm(x_i-y_i):i=1,\ldots,k\right\}=\Gamma^*$. Moreover, if $\left\{x_i+y_i:i=1,\ldots,k\right\}\subseteq\Gamma^*$ and $\left|\left\{x_i+y_i:i=1,\ldots,k\right\}\right|=q$, then $S$ is called \emph{strong starter} for $\Gamma$. There are some interesting results on strong starters for cyclic groups \cite{Avila,MR808085}, in particular for $\mathbb{Z}_n$ \cite{Skolem,AvilaSkolem,AvilaSkolem2,AvilaSkolem3} and $\mathbb{F}_q$ \cite{Avila,MR0325419,dinitz1984,MR0392622,MR0249314,MR0260604}, and for finite Abelian groups \cite{MR1010576,MR1044227}.	

Strong starters were first introduced by Mullin and Stanton in \cite{MR0234587} in constructing of Room
squares. Starters and strong starters have been useful to construct many combinatorial designs such as Room cubes \cite{MR633117}, Howell designs \cite{MR728501}, Kirkman triple systems \cite{MR808085,MR0314644}, Kirkman squares and cubes \cite{MR833796,MR793636},  and factorizations of complete graphs \cite{MR0364013,MR2206402,MR1010576,MR685627}.

Let $\Gamma$ be a finite additive Abelian group of odd order $n=2k+1$. It is well known that $F_\gamma=\{\{\infty, \gamma\}\}\cup\{\{x_i+\gamma,y_i+\gamma\} : 1 \leq i\leq k\}$, for all $\gamma\in\Gamma$, forms a one-factorization of the complete graph on $\Gamma\cup\{\infty\}$. Hence, if $F_0=\{\{\infty,0\}\}\cup\{\{x_i,y_i\} : 1 \leq i\leq k\}$, then $F_\gamma=F_0+\gamma$, for all $\gamma\in\Gamma^*$. On the other hand, let $S=\{\{x_i,y_i\}: 1\leq i\leq k\}$ and $T=\{\{u_i,v_i\}: 1\leq i\leq k\}$ be two starters for $\Gamma$. Without loss of generality, we assume $x_i-y_i=u_i-v_i$, for all $i=1,\ldots,k$. Then $S$ and $T$ are \emph{orthogonal starters} if $u_i-x_i=u_j-x_j$ implies $i=j$, and if $u_i\neq x_i$, for all $i=1,\ldots,q$.

Let $S=\{\{x_i,y_i\}:i=1,\ldots,k\}$ be a starter for a finite additive Abelian group $\Gamma$ of odd order $n=2k+1$. It is easy to see $-S=\{\{-x_i,-y_i\}:i=1,\ldots,k\}$ is also a starter for $\Gamma$. Also, the set of pairs $P = \{\{x,-x\}: x\in\Gamma^*\}$ is a starter for $\Gamma$, called \emph{the patterned starter}.

\begin{teo}\cite{MR623318}
If there is a strong starter $S$ in an additive Abelian group of odd order, then $S$, $-S$ and $P$ are pairwise orthogonal starters.		
\end{teo}

\begin{lema}\cite{DinitzSequentially}
For odd $n\geq5$, the one-factorization of $K_{n+1}$ generated by the patterned starter is $C_4$-free, if and only if, $n\not\equiv 0$ (mod 3). 	
\end{lema}

Let $q$ be an odd prime power. An element $x\in\mathbb{F}_q^*$ is called a \emph{quadratic residue} if there exists an element $y\in\mathbb{F}_q^{*}$ such that $y^2=x$. If there is no such $y$, then $x$ is called a \emph{non-quadratic residue.} The set of quadratic residues of $\mathbb{F}_q^{*}$ is denoted by $QR(q)$ and the set of non-quadratic residues is denoted by $NQR(q)$. It is well known $QR(q)$ is a cyclic subgroup of $\mathbb{F}_q^{*}$ of order $\frac{q-1}{2}$ (see for example \cite{MR2445243}). As well as, it is well known, if either $x,y\in QR(q)$ or $x,y\in NQR(q)$, then $xy\in QR(q)$. Also, if $x\in QR(q)$ and $y\in NQR(q)$, then $xy\in NQR(q)$. For more details of this kind of results the reader may consult \cite{burton2007elementary,MR2445243}.

An interesting (strong) starter for $\mathbb{F}_q$, the finte field of order prime power $q$. Horton in \cite{MR623318} proved (also proved by the author in \cite{Avila}) the following
\begin{prop}\cite{MR623318}
	If $q\equiv3$ (mod 4) is an odd prime power ($q\neq3$) and $\beta\in NQR(q)\setminus\{-1\}$, then 
	\begin{eqnarray*}\label{strong_1}
		S_\beta=\left\{\{x,x\beta\}:x\in QR(q)\right\},
	\end{eqnarray*} 
	is a strong starter for $\mathbb{F}_q$.
\end{prop}

In same paper,\cite{MR623318}, Horton proved the following

\begin{teo}\cite{MR623318}
Let $q\equiv3$ (mod 4) be a prime power ($q\neq3$). If $\beta_1,\beta_2\in NQR(q)\setminus\{-1\}$, with $\beta_1\neq\beta_2$, then $S_{\beta_1}$ and $S_{\beta_2}$ are
	orthogonal. 
\end{teo}

In this note, we prove the following theorem given by Bao and Ji in \cite{Bao}.

\begin{teo}\cite{Bao}\label{thm:inicio}
	Let $q\equiv$3 (mod 4) be an odd prime power with $q\geq11$, then there is a pair of orthogonal totally $C_4$-free one-factorizations  of $K_{q+1}$.
\end{teo}

The method utilized for constructing such a pair of orthogonal totally $C_4$-free one-factorizations is different than used in \cite{Bao}. Moreover, for the case when $q\equiv3$ (mod 8), our result generalizes Theorem 7 given by Bao and Ji in \cite{Bao}.


\section{Results}\label{sec:main}
In this section, we present some results of orthogonal $C_4$-free one-factorizations of the complete graph. Also, we present the main result of this note.

\begin{lema}\label{lemma:1_adrian}
	Let $q\equiv3$ (mod 4) be an odd prime power and $\beta\in NQR(q)\setminus\{-1\}$. If $\frac{\beta^2+1}{\beta-1}\in QR(q)$, then the one-factorization generated by the starter $S_\beta$ is $C_4$-free. 
\end{lema}

\begin{proof}
	Let $F_i=\{\{\infty,i\}\}\cup\{\{x+i,x\beta+i\}: x\in QR(q)\}$, for $i\in\mathbb{F}_q$. Then
	$\mathcal{F}=\{F_i:i\in\mathbb{F}_q\}$ is a one-factorization of the complete graph on $\mathbb{F}_q\cup \{\infty\}$. Assume  this one-factorization is not $C_4$-free. Since $F_i=F_0+i$,then without loss of generality, it is sufficient to prove $F_0\cup F_i$, for some $i\in\mathbb{F}_q$, contain a cycle of length 4, $C_4$. 
	
	First, suppose that $\infty\in V(C_4)$. Hence $\{\infty,0\},\{\infty,i\}\in E(C_4)$, where $\{\infty,0\}\in F_0$ and $\{\infty,i\}\in F_i$. We have $\{-i+i,-i\beta^{t_1}+i\},\{i,i\beta^{t_2}\}\in E(C_4)$, with $t_1=1$ if $-i\in QR(q)$, and $t_1=-1$ if $-i\in NQR(q)$, and $t_2=1$ if $i\in QR(q)$, and $t_2=-1$ if $i\in NQR(q)$. Since $F_0\cup F_i$ contain a cycle of length 4, then $\{\infty,0\}\{0,-i\beta^{t_1}+i\}\{i,i\beta^{t_2}\}\{\infty,i\}=C_4$. Therefore
	\begin{equation*}
	\beta^{t_2}+\beta^{t_1}-1=0, \mbox{where }t_1,t_2\in\{-1,1\}
	\end{equation*}
	Since $-i^2\in NQR(q)$, then $t_1\neq t_2$, which implies that $\beta^2-\beta+1=0$. If $\beta^2+1\in NQR(q)$, then $\beta-1\in NQR(q)$. Therefore, $\beta^2-\beta+1=0$ implies $\beta(\beta-1)=-1$, which is a contradiction, since $\beta(\beta-1)\in QR(q)$. On the other hand, if $\beta^2+1\in QR(q)$ then $\beta-1\in QR(q)$, hence $\beta^2-\beta+1=0$ implies $\beta^2+1=\beta$, which is a contradiction, since $\beta\in NQR(q)$.
	
	Now, suppose that $\infty\not\in V(C_4)$. Since $E(C_4)\cap F_0\neq\emptyset$, then without loss of generality there is $a\in V(C_4)\cap QR(q)$ such that
	\begin{center}
		$\{a,a\beta\}\in F_0$ and $\{a\beta,(a\beta-i)\beta^{t_1}+i\}\in F_i$,
	\end{center}
	where $t_1=1$ if $a\beta-i\in QR(q)$, and $t_1=-1$ if $a\beta-i\in NQR(q)$. On the other hand
	\begin{center}
		$\{a,(a-i)\beta^{t_2}+i\}\in F_i$ and $\{(a-i)\beta^{t_2}+i,((a-i)\beta^{t_2}+i)\beta^{t_3}\}\in F_0$,	
	\end{center}
	where $t_2=1$ if $a-i\in QR(q)$, and $t_2=-1$ if $a-i\in NQR(q)$, and $t_3=1$ if $(a-i)\beta^{t_2}+i\in QR(q)$, and $t_3=-1$ if $(a-i)\beta^{t_2}+i\in NQR(q)$. Therefore, we have
	\begin{equation}\label{equation}
	a(\beta^{t_1+1}-\beta^{t_2+t_3})+i(\beta^{t_2+t_3}-\beta^{t_3}-\beta^{t_1}+1)=0,
	\end{equation}
	for all $t_1,t_2,t_3\in\{-1,1\}$.
	\begin{itemize}
		\item [case (i)] Suppose $t_1=t_2=t_3=1$. In this case $\beta^2-2\beta+1=0$, which is a contradiction.
		\item [case (ii)] Suppose $t_1=t_2=1$ and $t_3=-1$. In this case $i=\frac{a\beta(\beta+1)}{\beta-1}$. Since $a-i\in QR(q)$, we have $a-i=-a\frac{\beta^2+1}{\beta-1}\in NQR(q)$, which is a contradiction, since $\frac{\beta^2+1}{\beta-1}\in QR(q)$ and $-a\in NQR(q)$.
		\item [case (iii)] Suppose $t_1=t_3=1$ and $t_2=-1$. In this case $i=\frac{a(\beta+1)}{2}$. We have $\frac{(a-i)\beta^{-1}+i}{a\beta-i}\in QR(q)$, since
		$a\beta-i\in QR(q)$ and $(a-i)\beta^{-1}+i\in QR(q)$. However $\frac{(a-i)\beta^{-1}+i}{a\beta-i}=\beta^{-1}\frac{\beta^2+1}{\beta-1}\in NQR(q)$, a contradiction. 
		\item [case (iv)] Suppose $t_2=t_3=1$ and $t_1=-1$. In this case $i=a\beta\frac{(\beta+1)}{\beta^2+1}$. Since $a-i\in QR(q)$, we have $a-i=-a\frac{\beta-1}{\beta^2+1}\in NRQ(q)$, which is a contradiction. 
		\item [case (v)] Suppose $t_1=1$ and $t_2=t_3=-1$. In this case $i=a(\beta+1)$. Since $a-i\in NQR(q)$, we have $a-i=-a\beta\in QR(q)$, a contradiction.
		\item [case (vi)] Suppose $t_2=1$ and $t_1=t_3=-1$. In this case $\beta-1=0$, which is a contradiction.
		\item [case (vii)] Suppose $t_3=1$ and $t_1=t_2=-1$. In this case $\beta^2-2\beta+1=0$, which is a contradiction.
		\item [case (viii)] Suppose $t_1=t_2=t_3=-1$. In this case $i=-a\frac{\beta+1}{\beta-1}$. Since $a\beta-i\in NQR(q)$, we have $a\beta-i=a\frac{\beta^2+1}{\beta-1}\in QR(q)$, a contradiction.
	\end{itemize} 
\end{proof}

\begin{example}
	For $q=19$, if $\beta\in\{2,3,10,13\}$, then $\frac{\beta^2+1}{\beta-1}\in QR(q)$. It is not difficult to check the one-factorization generated by the starter $S_\beta$ is $C_4$-free. 
\end{example}

\begin{lema}\label{lemma_Adrian_2}
	Let $q\equiv3$ (mod 4) be an odd prime power and $\beta\in NQR(q)\setminus\{-1\}$. If either $\frac{\beta^2+1}{2}\in QR(q)$ or $\frac{\beta-1}{2}\in QR(q)$, with $\beta\not\in\{2,\frac{1}{2}\}$, then the pair of orthogonal one-factorizations generated by the starter $S_\beta$ and $S_{-\beta}$ is $C_4$-free. 
\end{lema}
\begin{proof}
	Let $F_i=\{\{\infty,i\}\}\cup\{\{x+i,x\beta+i\}: x\in QR(q)\}$ and $G_i=\{\{\infty,i\}\}\cup\{\{-x+i,-x\beta+i\}: x\in QR(q)\}$, for $i\in\mathbb{F}_q$. Then
	$\mathcal{F}=\{F_i:i\in\mathbb{F}_q\}$ and $\mathcal{G}=\{G_i:i\in\mathbb{F}_q\}$ are orthogonal one-factorizations of the complete graph on $\mathbb{F}_q \cup\{\infty\}$. Assume this pair of one-factorizations is not $C_4$-free. Since $F_i=F_0+i$ and $G_i=G_0+i$, then without loss of generality we can suppose $F_0\cup G_i$, for some $i\in\mathbb{F}_q$, contain a cycle of length 4, $C_4$. 
	
	First, suppose that $\infty\in V(C_4)$. Hence, $\{\infty,0\},\{\infty,i\}\in E(C_4)$, where $\{\infty,0\}\in F_0$ and $\{\infty,i\}\in G_i$. Therefore, we have   $\{0,-i\beta^{t_1}+i\},\{i,i\beta^{t_2}\}\in E(C_4)$, with $t_1=t_2=1$ if $i\in QR(q)$, and $t_1=t_2=-1$ if $i\in NQR(q)$. Since $F_0\cup F_i$ contain a cycle of length 4, then $\{\infty,0\}\{0,-i\beta^{t_1}+i\}\{i,i\beta^{t_2}\}\{\infty,i\}=C_4$. Hence
	\begin{equation*}
	\beta^{t_2}+\beta^{t_1}-1=0, \mbox{where }t_1,t_2\in\{-1,1\}\mbox{ with $t_1=t_2=t$},
	\end{equation*}
	which implies that $2\beta^t=1$, which is a contradiction if $\beta\not\in\{2,\frac{1}{2}\}$.
	
	Now, suppose $\infty\not\in V(C_4)$. Since $-1\in NQR(q)$, then $NQR(q)=-QR(q)$. Therefore, $G_i=\{\{y+i,y\beta+i\}:y\in NQR(q)\}$. Since $E(C_4)\cap F_0\neq\emptyset$, without loss of generality there is $a\in V(C_4)\cap QR(q)$ such that
	\begin{center}
		$\{a,a\beta\}\in F_0$ and $\{a\beta,(a\beta-i)\beta^{t_1}+i\}\in G_i$,
	\end{center}
	where $t_1=1$ if $a\beta-i\in NQR(q)$, and $t_1=-1$ if $a\beta-i\in QR(q)$. 
	
	On the other hand
	\begin{center}
		$\{a,(a-i)\beta^{t_2}+i\}\in G_i$ and $\{(a-i)\beta^{t_2}+i,((a-i)\beta^{t_2}+i)\beta^{t_3}\}\in F_0$,	
	\end{center}
	where $t_2=1$ if $a-i\in NQR(q)$, and $t_2=-1$ if $a-i\in QR(q)$, and $t_3=1$ if $(a-i)\beta^{t_2}+i\in QR(q)$, and $t_3=-1$ if $(a-i)\beta^{t_2}+i\in NQR(q)$.
	
	Therefore, we have
	\begin{equation}
	a(\beta^{t_1+1}-\beta^{t_2+t_3})+i(\beta^{t_2+t_3}-\beta^{t_3}-\beta^{t_1}+1)=0,
	\end{equation}
	for all $t_1,t_2,t_3\in\{-1,1\}$, which is an analogous equation obtained in the proof of Lemma \ref{lemma:1_adrian}.
	
	\begin{itemize}
		\item [case (i)] Suppose $t_1=t_2=t_3=1$. In this case $\beta^2-2\beta+1=0$, which is a contradiction.
		
		\item [case (ii)] Suppose $t_1=t_2=1$ and $t_3=-1$. In this case $i=a\frac{\beta(\beta+1)}{\beta-1}$. We have that $a\beta-i\in NQR(q)$, but  $a\beta-i=-a\beta\frac{2}{\beta-1}\in QR(q)$, since $-a\beta\in QR(q)$, which is a contradiction. On the other hand, since $a-i\in NQR(q)$ and $a\beta-i\in NQR(q)$, then $\frac{a-i}{a\beta-i}\in QR(q)$, but $\frac{a-i}{a\beta-i}=\frac{\beta^2+1}{2\beta}\in NQR(q)$, which is a contradiction.
		
		\item [case (iii)] Suppose $t_1=t_3=1$ and $t_2=-1$. In this case $i=a\frac{\beta+1}{2}$. We have $a\beta-i\in NQR(q)$, but $a\beta-i=a\frac{\beta-1}{2}\in QR(q)$, which is a contradiction. On the other hand, we have $(a-i)\beta^{-1}+i\in QR(q)$, but $(a-i)\beta^{-1}+i=\frac{a(\beta^2+1)}{2\beta}\in NQR(q)$, which is a contradiction, since $\frac{a}{\beta}\in NQR(q)$.
		
		\item [case (iv)] Suppose $t_2=t_3=1$ and $t_1=-1$. In this case $i=a\beta\frac{(\beta+1)}{\beta^2+1}$. Since $(a-i)\beta+i\in QR(q)$, we have $(a-i)\beta+i=a\beta\frac{2}{\beta^2+1}\in NRQ(q)$, which is a contradiction. On the other hand, since $a-i\in NQR(q)$ and $(a-i)\beta+i\in QR(q)$, then $\frac{a-i}{(a-i)\beta+i}\in NQR(q)$, but $\frac{a-i}{(a-i)\beta+i}=-\frac{\beta-1}{2\beta}\in QR(q)$, which is a contradiction, since $-\beta^{-1}\in QR(q)$. 
		
		\item [case (v)] Suppose $t_1=1$ and $t_2=t_3=-1$. In this case $i=a(\beta+1)$. If there is a cycle of lenght 4, then $a\neq (a\beta-i)\beta+i=((a-i)\beta^{-1}+i)\beta^{-1}$, but $(a\beta-i)\beta+i=a=((a-i)\beta^{-1}+i)\beta^{-1}$, which is a contradiction.
		
		\item [case (vi)] Suppose $t_2=1$ and $t_1=t_3=-1$. In this case $\beta-1=0$, which is a contradiction.
		
		\item [case (vii)] Suppose $t_3=1$ and $t_1=t_2=-1$. In this case $\beta^2-2\beta+1=0$, which is a contradiction.
		
		\item [case (viii)] Suppose $t_1=t_2=t_3=-1$. In this case $i=-a\frac{\beta+1}{\beta-1}$. We have $a-i\in QR(q)$, but $a-i=a\beta\frac{2}{\beta-1}\in NQR(q)$, wich is a contradiction. On the other hand, since $a\beta-i\in QR(q)$ and $a-i\in QR(q)$, then $\frac{a\beta-i}{a-i}\in QR(q)$, but $\frac{a\beta-i}{a-i}=\frac{\beta^2+1}{2\beta}\in NQR(q)$, which is a contradiction.
	\end{itemize} 
\end{proof}

\subsection{Main result}

Let $q=ef+1$ be a prime power and let $H$ be the subgroup of $\mathbb{F}_q^*$ of order $f$ with $\{H=C_0,\ldots,C_{e-1}\}$ the set of (multiplicative) cosets of $H$ in $\mathbb{F}_q^*$ (that is, $C_i = g^iC_0$, where $g$ is the least primitive element of $\mathbb{F}_q$). The \emph{cyclotomic number}
$(i,j)$ is defined as $|\{x\in C_i: x+1\in C_j\}|$. In particular, if $e=2$ and $f$ is odd, then $(0,0)=\frac{f-1}{2}$, $(0,1)=\frac{f+1}{2}$, $(1,0)=\frac{f-1}{2}$ and $(1,1)=\frac{f-1}{2}$, see \cite{book:206537}, Table VII.8.50. Hence, if $C_0=QR(q)$ and $C_1=NQR(q)$ are the cosets of $QR(q)$ in $\mathbb{F}_q^*$, then the following is satisfied:

\begin{lema}\label{lemma:1+beta}
	Let $p\equiv3$ (mod 4) be an odd prime power with $p\neq3$. Then, there exists $\beta_1,\beta_2\in NQR(q)$ such that 
	\begin{enumerate}
		\item $(\beta_1+1)\in NQR(q)$ and $(\beta_2+1)\in QR(q)$.
		\item $(\beta_1-1)\in NQR(q)$ and $(\beta_2-1)\in QR(q)$.
	\end{enumerate}
\end{lema}

Which implies too that:

\begin{lema}\label{lemma:1+alpha}
	Let $p\equiv3$ (mod 4) be an odd prime power with $p\neq3$. Then there are $\alpha_1,\alpha_2\in QR(q)$ such that 
	\begin{enumerate}
		\item $(\alpha_1+1)\in NQR(q)$ and $(\alpha_2+1)\in QR(q)$.
		\item $(\alpha_1-1)\in NQR(q)$ and $(\alpha_2-1)\in QR(q)$.
	\end{enumerate}
\end{lema}

The following Lemma \ref{lemma:NQR} is analogous then given in \cite{Avila}, and we give the proof given in the same paper.

\begin{lema}\label{lemma:NQR}
	Let $p\equiv3$ (mod 4) be an odd prime power with $p\neq3$. Then there is $\beta\in NQR(q)$ such that $(\beta+1)(\beta-1)\in NQR(q)$.
\end{lema}
\begin{proof}
	For each $\beta\in NQR(q)$ define $A_\beta=\{a^\beta_1,\ldots,a^\beta_{l_\beta}\}$, where $a^\beta_{i}\in NQR(q)$ and $a^\beta_{i+1}=a^\beta_{i}+1$, for all $i=1, \ldots,l_{\beta-1}$ with $\beta=a_1^\beta$. By Lemma \ref{lemma:1+beta}, there is $\beta\in NQR(q)$ such that $|A_\beta|>1$. If $\beta^*=a^\beta_{l_\beta}$ then $(\beta^*+1)\in NQR(q)$ and $(\beta^*-1)\in QR(q)$.
	On the other hand, if $\beta^*=\beta_1$ then $(\beta^*+1)\in QR(q)$ and $(\beta^*-1)\in NQR(q)$. Hence $(\beta^*+1)(\beta^*-1)\in NQR(q)$.
\end{proof}

\begin{lema}\label{lemma:QR}
	Let $p\equiv3$ (mod 4) be an odd prime power with $p\neq3$. Then there is $\alpha\in QR(q)$ such that $(\alpha+1)(\alpha-1)\in NQR(q)$.
\end{lema}
\begin{proof}
	For each $\alpha\in QR(q)$ define $A_\alpha=\{a^\alpha_1,\ldots,a^\alpha_{l_\alpha}\}$, where $a^\alpha_{i}\in QR(q)$ and $a^\alpha_{i+1}=a^\alpha_{i}+1$, for all $i=1, \ldots,l_{\alpha-1}$ with $\alpha=a_1^\alpha$. By Lemma \ref{lemma:1+beta}, there is $\alpha\in QR(q)$ such that $|A_\alpha|>1$. If $\alpha^*=a^\alpha_{l_\alpha}$ then $(\alpha^*+1)\in NQR(q)$ and $(\alpha^*-1)\in QR(q)$.
	On the other hand, if $\alpha^*=\alpha_1$ then $(\alpha^*+1)\in QR(q)$ and $(\alpha^*-1)\in NQR(q)$. Hence $(\alpha^*+1)(\alpha^*-1)\in NQR(q)$.
\end{proof}

The following theorem is the main theorem of this note. 

\begin{teo}\label{teo:main_main}
	Let $q\equiv$3 (mod 4) be an odd prime power with $q\geq11$, then there is a pair of orthogonal totally $C_4$-free one-factorizations  of $K_{q+1}$.
\end{teo}

\begin{proof}
	Let $M=M_1\cup M_2$, where
	\begin{center}
		$M_1=\left\{\beta\in NQR(q)\setminus\{2,2^{-1},-1\}:\frac{\beta^2+1}{2}\in QR(q),\beta^3\neq-1\right\},$
		
		$M_2=\left\{\beta\in NQR(q)\setminus\{2,2^{-1},-1\}:\frac{\beta-1}{2}\in QR(q),\beta^3\neq-1\right\}$,	
	\end{center}
	By Lemma \ref{lemma:1_adrian} and Lemma \ref{lemma_Adrian_2}, we need only show that $M$ is not the empty set:
	Since $-1\in NQR(q)$ then $\beta\to\beta^2$, for all $\in\beta\in NQR(q)$, is a bijection between $NQR(q)$ and $QR(q)$. By Lemma \ref{lemma:QR} there is $\beta\in NQR(q)$ such that $(\beta^2+1)(\beta^2-1)\in NQR(q)$. Furthermore, there are $\beta_1,\beta_2\in NQR(q)$ such that $\beta^2_1-1\in NQR(q)$ and $\beta^2_1+1\in QR(q)$, and $\beta^2_2-1\in QR(q)$ and $\beta^2_2+1\in NQR(q)$, see proof of Lemma \ref{lemma:QR}.
	
	If $2\in QR(q)$ then let $\beta\in NQR(q)$ such that $\beta^2+1\in QR(q)$ and $\beta^2-1\in NQR(q)$, by Lemma \ref{lemma:QR}. We assume that $\beta^3\neq-1$, since if $\beta^3+1=0$ then $\beta^2-\beta+1=0$, which implies that $\beta^2+1=\beta$, a contradiction, since $\beta^2+1\in QR(q)$ and $\beta\in NQR(q)$. Hence $M_1\neq\emptyset$. Since $(\beta^2-1)=(\beta-1)(\beta+1)\in NQR(q)$, then either $\beta-1\in QR(q)$ and $\beta+1\in NQR(q)$, which implies that $M_2\neq\emptyset$, or $\beta-1\in NQR(q)$ and $\beta+1\in QR(q)$. On the other hand, if $2\in NQR(q)$ then let $\beta\in NQR(q)$ such that $\beta-1\in NQR(q)$ and $\beta+1\in QR(q)$, see proof of Lemma \ref{lemma:NQR}. We assume that $\beta^3\neq-1$, since if $\beta^3+1=0$ then $\beta^2-\beta+1=0$, which implies that $\beta(\beta-1)=-1$, which is a contradiction, since $\beta(\beta-1)\in QR(q)$. Since $\beta-1\in NQR(q)$ then $\beta\neq2$. Hence $M_2\neq\emptyset$. Finally, if $\beta^2+1\in NQR(q)$ then $M_1\neq\emptyset$.\qed   
\end{proof}

It is well known, if $q\equiv3$ (mod 8), then $2\in NQR(q)$ (see for example \cite{Ireland}). Therefore

\begin{coro}\label{coro:final}
	Let $q\equiv$3 (mod 8) be an odd prime power with $q\geq11$. If $M=M_1\cup M_2$, where
	\begin{eqnarray*}
		M_1&=&\left\{\beta\in NQR(q)\setminus\{2,2^{-1},-1\}:\beta^2+1\in NQR(q),\beta^3\neq-1\right\}\\
		&&\\
		M_2&=&\left\{\beta\in NQR(q)\setminus\{2,2^{-1},-1\}:\beta-1\in QR(q),\beta^3\neq-1\right\},	
	\end{eqnarray*}
	then the one-factorization generated by the starters $S_{\beta_1}$ and $S_{\beta_2}$, for any pair of different elements $\beta_1,\beta_2\in M$, are totally $C_4$-free of $K_{q+1}$.	
\end{coro}

Hence, for the case when $q\equiv3$ (mod 8), Corollary \ref{coro:final} generalizes Theorem 7 given by Bao and Ji in \cite{Bao}.

\begin{teo}\cite{Bao}
	Let $q\equiv$3 (mod 8) be an odd prime power with $q\geq11$. If 
	\begin{center}
		$M=\left\{\beta\in NQR(q)\setminus\{2,2^{-1},-1\}:\beta^2+1\in NQR(q),\beta^3\neq-1\right\},$	
	\end{center}
	then the one-factorization generated by the starters $S_{\beta_1}$ and $S_{\beta_2}$, for any pair of different elements $\beta_1,\beta_2\in M$, are totally $C_4$-free of $K_{q+1}$.
\end{teo}

{\bf Acknowledgment}

Research was partially supported by SNI and CONACyT.

\end{document}